\documentclass[11pt]{amsart}

\newtheorem{theorem}{Theorem}[section]

\newtheorem{corollary}[theorem]{Corollary}
\newtheorem{proposition}[theorem]{Proposition}
\newtheorem{lemma}[theorem]{Lemma}

\theoremstyle{definition}
\newtheorem{definition}[theorem]{Definition}

\newtheorem{example}[theorem]{Example}

\numberwithin{equation}{section}

\newcommand{\F}{\mathbb{F}}

\newcommand{\define}[1]{\emph{#1}}
\newcommand{\cS}{\mathcal{S}}

\newcommand{\set}[1]{\left\{#1\right\}}

\newcommand{\setof}[2]{\set{#1 \thinspace : \thinspace #2}}

\begin{document}

\title{Projection-Forcing Multisets of Weight Changes}
\author{Josh Brown Kramer \and Lucas Sabalka}

\begin{abstract}

Let $\F$ be a finite field.  A multiset $S$ of integers is
\emph{projection-forcing} if for every linear function $\phi
: \F^n \to \F^m$ whose multiset of weight changes is $S$, $\phi$ is a
coordinate projection up to permutation and scaling of entries.  The MacWilliams Extension
Theorem from coding theory says that $S = \{0, 0, \ldots, 0\}$ is projection-forcing.  We give a
(super-polynomial) algorithm to determine whether or not a given $S$ is
projection-forcing.  We also give a condition that can be checked in polynomial time that implies that $S$ is projection-forcing.  This result is a generalization of the MacWilliams Extension Theorem and work by the first author.

\end{abstract}

\maketitle

\section{Introduction}\label{sec:intro}

Let $\F = \F_q$ be a finite field of prime power order $q$.  In coding theory, two of the most important aspects of a subspace of $\F^n$ (also called a linear code) are its structure as a vector space and its \emph{weight distribution}, to be defined shortly.  In this paper, we look at the interplay between these two aspects, by determining the structure of some linear maps from their effects on weight distributions.  In this section we make our results more precise, and discuss previous results of this type.

We begin with some notation.

\begin{definition}[Hamming weight, weight distribution]
     Let $\F$ be a finite field, and let $V \subset \F^n$ be a linear code.  The \define{Hamming weight} of a vector $v \in V$, denoted $w(v)$, is the number of nonzero entries of $v$. The \define{weight distribution} of $V$ is the multiset of Hamming weights of elements of $V$.
\end{definition}

\begin{definition}[multiset of weight changes]
    Let $U \subseteq \F^n$ and $V \subseteq \F^m$ be linear codes, and let $\phi : U \to V$ be linear.  The \define{multiset of weight changes} of $\phi$ is the multiset
        $$\setof{w(u) - w(\phi(u))}{u \in U}.$$
\end{definition}

\begin{definition}[weight-preserving linear functions]
    Let $V \subseteq \F^n$ and $W \subseteq \F^m$ be linear codes.  We say that a linear function
$\phi : V \to W$ is \define{weight-preserving} if for all $v \in V$, we have $w(\phi(v)) = w(v)$.  Equivalently, $\phi$ is weight-preserving if its multiset of weight changes is $\{0,0,\ldots,0\}$.
\end{definition}

The MacWilliams  Extension Theorem \cite{MacWilliams} says that any weight-preserving linear function simply reorders entries and scales by nonzero constants.  To be more precise, we give the following definition.

\begin{definition}[monomial equivalence]
Let $V,W \subseteq \F^n$ be linear codes.  A linear function $\phi : V \to W$ is said to be a
\define{monomial equivalence} if $\phi$ is multiplication by an $\F$-valued $n \times
n$ matrix with exactly one nonzero entry in each row and column.
\end{definition}

It is clear that a monomial equivalence is weight-preserving.  In \cite{MacWilliams}, MacWilliams proved the converse:

\begin{theorem}[MacWilliams Extension Theorem \cite{MacWilliams, bgg1978}]
\label{thm:MacWilliams}
  Let $\F$ be a finite field, and let $U,V \subseteq \F^n$ be linear codes.
  A linear function $\phi : U \to V$ is weight-preserving if and only if it is a monomial equivalence.
\end{theorem}

In this paper we will generalize the MacWilliams Extension Theorem by determining the structure of some linear functions with multisets of weight changes other than $\{0,0,\ldots,0\}$.

The multiset of weight changes has some redundant information.  The weight change associated to the 0 vector is always 0.  Furthermore, for any nonzero scalar $\alpha \in \F$, the weight change of $v$ is the same as that of $\alpha v$.  For this reason, and because some statements become easier to make, we introduce the projective multiset of weight changes for a linear map and the projective weight distribution of a linear code.

\begin{definition}[projective]
  Let $V \subseteq \F^n$ be a linear code.  Define the \define{projective space} $P(V)$ to be $(V \setminus \{0\}) / \sim$, where $v_1 \sim v_2$ if $v_1 = \alpha v_2$ for some $\alpha \in \F \setminus \{0\}$.  Define the \define{projective geometry}, $PG(k-1,q)$, to be the projective space $P(\F_q^k)$.
  
  The \define{projective weight distribution} of $V$ is the multiset:
  	$$\setof{w(v)}{v \in P(V)},$$
where $w(v)$ is the common weight of the vectors in the equivalence class $v$.

Given linear codes $U,V \subseteq \F^m$ and linear map $\phi : U \to V$, the  \define{projective multiset of weight changes} of $\phi$ is the multiset:
        $$\setof{w(u) - w(\phi(u))}{u \in P(U)}.$$
\end{definition}

Notice that if $S$ is the projective multiset of weight changes for a linear map from a $k$-dimensional code of $\F_q^n$, then $|S| = (q^{k} - 1)/(q - 1)$.

In \cite{bk07}, the first author proved a generalization of the MacWilliams Extension Theorem that can be expressed in terms of the projective multiset of weight changes.  To state this result, we need the following definition.

\begin{definition}[projection]
Let $U \subseteq \F^n$ and $V \subseteq \F^m$ be linear codes.  A linear function $\phi : V \to W$ is said to be a \define{coordinate projection up to monomial equivalence} if $\phi$ is multiplication by an $\F$-valued matrix with at most one nonzero entry in each row and column.  Throughout the paper we will simply call such a function a \define{projection}.
\end{definition}

\begin{theorem} \cite{bk07}
    Let $\F$ be a finite field, and let $U \subseteq \F^n$ and $V \subseteq \F^m$ be linear codes.  If the projective multiset of weight changes of a linear function $\phi : U \to V$ is $\set{c,c,\ldots,c}$, where $c$ is a constant, then $\phi$ is a projection.
\end{theorem}
Brown Kramer used this result to study a problem in extremal combinatorics posed in \cite{realspace}.

 The results of MacWilliams and Brown Kramer show that it is sometimes possible to determine that a map is a projection simply by knowing its multiset of weight changes.   The goal of this paper is to better understand for which multisets this is true.  To that end, we introduce the following definitions:
 
\begin{definition}[realizes, projection-forcing]
    If $S$ is the projective multiset of weight changes of $\phi$, we say that $\phi$
    \define{realizes} $S$.  If $S$ is such that every $(\F_q)$-linear map that realizes $S$ is a projection,
    we say that $S$ is \define{projection-forcing} or more explicitly,
    \define{$q$-projection-forcing}.
\end{definition}

One might think that if $S$ is realized by a projection, then $S$ is projection-forcing.  The following small example shows that this is false.  Thus we cannot always determine whether or not a function is a projection by looking at its multiset of weight changes.

\begin{example}\label{example:sameSet}

We present two linear maps, $\phi_1$ and $\phi_2$.  They have the same
multiset of weight changes, but one of them is a projection, and the other is not.

Let $V_1 \subseteq \F_2^7$ be the vector space generated by
$$\{(1,1,1,1,0,0,0),(1,1,0,0,1,1,0),(1,0,1,0,1,0,1)\}.$$

Let $V_2 \subseteq \F_2^7$ be the vector space generated by
	$$\{(1,1,1,1,0,0,0),(1,1,1,0,1,0,0),(1,1,0,0,0,1,1)\}.$$

Let $W \subseteq \F_2^2$ be $\{(0,0),(1,1)\}$.

Define $\phi_1: V_1 \to W$ to be the map that  makes two copies of the first coordinate.

Define $\phi_2: V_2 \to W$ to be the projection onto the first two coordinates of $V_2$.

The projective multiset of weight changes for each map is $\{2,2,2,2,4,4,4\}$.  However, $\phi_1$ cannot be a coordinate projection since there is no pair of coordinates where the elements of $V_1$ are always equal. On the other hand, $\phi_2$ is explicitly a coordinate projection.
\end{example}

We have two main results that help determine which multisets are projection-forcing.  The first result gives a property that characterizes these multisets.  It is time consuming to determine if a given multiset has this property.  It requires the following matrix.

\begin{definition}[$M_{k,q}$]
	For a nonnegative integer $k$ and a prime power $q$, define the real-valued matrix $M_{k,q}$ to be the complement of the incidence matrix of the design formed by hyperplanes and points in $PG(k-1,q)$.
\end{definition}

We will give a more explicit construction of $M_{k,q}$ in the proof of the following theorem.

\begin{theorem}
	\label{main2}
Let $k$ be a nonnegative integer, and $q$ a prime power.  A multiset $S$ of size $(q^{k} - 1)/(q - 1)$ is projection-forcing if and only if for each vector $\pi$, a permutation of $S$, either $M_{k,q}^{-1} \pi$ is nonnegative or it contains a non-integer entry.
\end{theorem}
Equipped with this matrix, one can check in super-exponential time that a set is projection-forcing.

The second of our main results gives a property that can be checked quickly and that implies, but is not equivalent to, projection-forcing.  It uses a technical, but easily-computed parameter $\delta_q(S)$ that we call the split difference.  In the binary case the split difference is simply the sum of the smallest $2^{k-1}$ elements of $S$ minus the sum of the largest $2^{k-1}-1$ elements of $S$, hence the name.  We put off the general definition until Section \ref{sec:qary}.

\begin{theorem}
	\label{main1}
  If $S$ is a multiset of size $(q^{k} - 1)/(q - 1)$ and $\delta_q(S) > -q^{k-1}$ then $S$ is $q$-projection-forcing.
\end{theorem}

We will see that this second theorem is a generalization of Brown Kramer's theorem, and hence the MacWilliams Extension Theorem.

In Section \ref{sec:qary}, we prove our main results.  We conclude in Section \ref{sec:Other} by giving miscellaneous results that potentially give insight into a full characterization of those sets that force projections.

\section{Proof of the Main Results}\label{sec:qary}

First we give some notation.  Let $q$ be a prime power, and let $k$ be a positive integer.  Define $G_{k,q}$ to be a matrix whose columns are representatives of the elements of $PG(k-1,q)$.  Define $\cS_{k,q}$ to be the row space of $G_{k,q}$.  Incidentally, the code $\cS_{k,q}$ is the \define{q-ary simplex code of dimension $k$}, the dual of a Hamming code.  For more on these topics see, for example, \cite{hp2003}.

We are now ready to prove Theorem \ref{main2}.

\begin{proof}[Proof of Theorem \ref{main2}]
Define $M_{k,q}$ to be the complement of the incidence matrix of the design formed by hyperplanes and points in $PG(k-1,q)$.  That is to say, index the rows and columns of $M_{k,q}$ by elements of $PG(k-1,q)$; given $v,w \in \F_q\mathbb{P}^k$, the $(v,w)^{th}$ entry of $M$ is 0 if $v \cdot w = 0$ and 1 otherwise.  As an alternative definition, let the rows of $M_{k,q}$ be the vectors of $P(\cS_{k,q})$ and change nonzero entries to real 1s and 0 entries to real 0s.  Notice that by the first construction, $M_{k,q}$ can be taken to be symmetric, but that we don't assume this symmetry in this paper.  We will show in Proposition \ref{prop:q-ary M inverse} that $M_{k,q}$ is invertible.

Suppose that $S$ is a multiset such that for every permutation $\pi$ of $S$, either $M_{k,q}^{-1}\pi$ is nonnegative, or it contains a non-integer.  Let $\phi : V \to W$ be a linear map that realizes $S$.  We want to show that $\phi$ is a projection.

Let $B$ be a generator matrix for $V$: its rows, $v_1,v_2,\ldots,v_k$ form a basis for $V$.  Given a point $p$ of $PG(k-1,q)$, define the \define{multiplicity of $p$ with respect to $v_1,v_2,\ldots,v_k$} to be the number of columns, $c$, of $B$ such that $c \sim p$.  Let $R$ be the vector of multiplicities, indexed by elements of $PG(k-1,q)$ (in the same order as for $M_{k,q}$).  Notice that $M_{k,q}R$ is the projective weight distribution of $V$.

Define $Q$ to be the vector of multiplicities with respect to $\phi(v_1),\ldots, \phi(v_k).$
Define $\pi = M_{k,q}R - M_{k,q}Q$.  Notice that $\pi$ is a permutation of $S$.  But $M_{k,q}^{-1}\pi = R - Q$ consists of integers, so by our choice of $S$, it is nonnegative.  This tells us that no multiplicity has increased, so $\phi$ is a projection.

Conversely, suppose that there is a permutation $\pi$ of $S$ for which $M_{k,q}^{-1} \pi$ consists of all integers, some of which are negative.  Then $M_{k,q}^{-1}\pi$ consists of the multiplicity differences for a linear function that realizes $S$. Since some of these differences are negative, the function is not a projection.
\end{proof}

It should be noted that Bogart, Goldberg, and Gordon \cite{bgg1978} gave a proof of the MacWilliams extension theorem by establishing the invertibility of $M_{k,q}$.
We will prove Theorem \ref{main1} by explicitly constructing an inverse for $M_{k,q}$ and using its structure.  The following formula is probably not a new result.

\begin{proposition}\label{prop:q-ary M inverse}
  The inverse of $M_{k,q}$ is
    $$\frac{1}{q^{k-1}}(qM_{k,q}^T - (q - 1)J),$$
  where $J$ is the all ones matrix.
\end{proposition}

\begin{proof}Call this alleged inverse $M'$.  Notice that $M'$ comes from $M_{k,q}^{T}$ by replacing each 1 with $1/q^{k-1}$ and each 0 with $-(q-1)/q^{k-1}$. Let $(i,j)$ index an entry of $M_{k,q} M'$. Every element of $\cS_{q,k}$ has Hamming weight $q^{k-1}$.
Thus if $i = j$, then the $(i,j)^{th}$ entry of $M_{k,q} M'$ is $q^{k-1}/q^{k-1} = 1$.

Now consider $i \neq j$.  We claim that the ones in row $i$ and row $j$ of $M_{k,q}$ overlap at $q^{k-1} - q^{k - 2}$ positions.  This is true because this overlap is the number of points of $PG(k - 1,q)$ not in the union of two hyperplanes.  Thus the $(i,j)^{th}$ entry of $M_{k,q} M'$ is $(q^{k-1} - q^{k - 2}) /q^{k-1} - (q^{k-1} - (q^{k-1} - q^{k - 2}))(q - 1)/q^{k-1} = 0.$

\end{proof}

We're nearly ready to prove Theorem \ref{main1} as a corollary of Proposition \ref{prop:q-ary M inverse}.  First we need to give the full definition of split difference.  The idea is to consider the smallest value possible in any entry of $M_{k,q}^{-1} \pi$, where $\pi$ is a permutation of $S$.  If this is positive, then Theorem \ref{main2} proves that $S$ is projection-forcing.  In light of the explicit structure of $M_{k,q}^{-1}$ given in Proposition \ref{prop:q-ary M inverse}, it is easy to compute this smallest value.  We almost define the split difference to be this number, but for aesthetic purposes we scale by $q^{k-1}$.

\begin{definition}[split difference]\label{def:qsplit}
Let $S = \{s_1,\dots, s_{(q^{k} - 1)/(q - 1)}\}$ be a multiset where $s_i \leq
s_{i+1}$ for all $i$.  Define the \define{$q$-ary split difference}, $\delta_q(S)$, of $S$ to be
  $$\delta_q(S) = \sum_{i=1}^{q^{k-1}} s_i - \left[(q-1)\sum_{i=q^{k-1}+1}^{(q^k - 1)/(q-1)} s_i\right].$$
\end{definition}

We can now prove Theorem \ref{main1}.

\begin{proof}[Proof of Theorem \ref{main1}]

Suppose $\delta_q(S) > -q^{k-1}$.
The smallest an entry of $M_{k,q}^{-1} \pi$ can be,
where $\pi$ is a permutation of $S$, is $\delta_q(S)/q^{k - 1} > -1$.  Thus if $M_{k,q}^{-1} \pi$ consists of
integers, it is nonnegative.  By Theorem \ref{main2}, $S$ is $q$-projection-forcing.

\end{proof}

The Brown Kramer and MacWilliams results follow:

\begin{corollary}

  If $S = \{c,c,\ldots,c\}$, where $c$ is positive, then $S$ is projection-forcing.

\end{corollary}

\begin{proof}

The $q$-ary split difference in this case is

        $$\delta_q(S) = cq^{k-1} - (q -
        1)\left(c\frac{q^{k-1} - 1}{q - 1}\right) =
        c \geq 0 > -q^{k - 1}.$$

\end{proof}

As a quick example of a realizable projection-forcing multiset
not covered by the MacWilliams or Brown Kramer results, consider $S = \set{3,3,3,4,4,4,7}$.  By Theorem \ref{main1}, $S$ is 2-projection-forcing, since $\delta_2(S) = 3+3+3+4-4-4-7=-2 > -2^{3 - 1}$.

As the next example shows, not all projection-forcing sets have the property from Theorem \ref{main1}:

\begin{example}

Consider $S = \{2,2,2,3,5,5,5\}$.  This is realized by many projections.  For instance, the projection onto the last coordinate of the space generated by the vectors $$\{(1,1,0,0,0,0,0),(1,0,1,0,0,0,0),(0,0,0,1,1,1,0)\}.$$  Although $\delta_2(S) = 2+2+2+3-5-5-5 = -6 \leq -2^{3-1}$, we claim $S$ is still projection-forcing.  We leave it to the reader to verify that if $S$ is the multiset of weight changes of a linear map, then the vector of changes in multiplicities must be some reordering of
$(0,0,0,1,1,1,3)$.  This statement can be checked via computer.  

Using a very short Mathematica program that uses Theorem \ref{main1}, we have verified that there are 58 projection-forcing multisets for $3$-dimensional binary vector spaces with weight changes at most 7.  Among these, the following 8 are the ones which are not caught by Theorem \ref{main1}: \\
$\{2,2,2,3,5,5,5\}$, $\{2,2,2,5,5,5,7\}$, $\{2,2,2,5,7,7,7\}$, $\{2,2,4,7,7,7,7\}$,\\
$\{2,3,3,3,5,6,6\}$, $\{2,4,4,5,7,7,7\}$, $\{3,3,3,4,6,6,7\}$, $\{3,4,4,4,7,7,7\}$.

\end{example}

\section{Other results}\label{sec:Other}

Brown Kramer's result handles the case when $S$ is constant.  In this section we deal with the binary case of some almost-constant multisets.  This might give insight into a characterization that is checked more efficiently than the one given in Theorem \ref{main2}.
First we consider the case where there is one discrepancy from being constant.  We determine the realizable multisets $S$ and then the projection-forcing realizable multisets $S$.

\begin{lemma}\label{lem:linearmapab}

Let $a$ and $b$ be nonnegative integers, and let $k \geq 2$.  Let $S$
consist of $2^k - 2$ copies of $a$ and one copy of $b$.  Then $S$ is
realized by some binary linear map if and only if $a \equiv 0 \mod
2^{k-2}$ and $b \equiv 0 \mod 2^{k-1}$.

\end{lemma}

\begin{proof}

Let $S$ be realized by some binary linear map. Let $M_{k,2}$ be the matrix as
defined in Section \ref{sec:qary}.  Let $\pi$ be a permutation of
$S$.  Every entry of $M_{k,2}^{-1} \pi$ is either $b/2^{k-1}$ or
$(2a - b)/2^{k-1}$. Furthermore, both values appear in
$M_{k,2}^{-1}\pi$ at least once.  Since one of these permutations corresponds to
a map, these numbers must be integers.  Thus the conditions on parity
hold.

Conversely, if the parity conditions hold, then let $\pi$ be a permutation
of $S$.  $M_{k,2}^{-1} \pi$ consists of integers, so we may use this product to
construct a linear map that realizes $S$.

\end{proof}

\begin{proposition}

Let $a$ and $b$ be nonnegative integers, and let $k \geq 2$.  Let $S$
consist of $2^k - 2$ copies of $a$ and one copy of $b$.  Suppose $S$ is realized by an $\F_2$-linear map.  Then the following are
equivalent:
\begin{enumerate}

\item $S$ is $2$-projection-forcing,
\item $b \leq 2a$,
\item $b < 2a+2^{k-1}$.

\end{enumerate}

\end{proposition}

\begin{proof}

Clearly, (2) implies (3).  Proposition \ref{lem:linearmapab} tells
us that (3) implies (2).  The split difference of $S$ is $b$ when
$0 \leq b \leq a$ and $2a-b$ when $a \leq b$.  If (3) holds, then either $\delta_2(S) = 2a - b > -2^{k-1}$ or $\delta_2(S) = b \geq 0 > -2^{k-1}$.
Thus Theorem \ref{main1} tells us that (3) implies (1).

We finish by proving (1) implies (2).  Suppose $S$ is
projection-forcing.  If $\pi$ is a permutation of $S$ then, as in the
proof of Proposition \ref{lem:linearmapab}, the entries of $M_{k,2}^{-1}
\pi$ are $b/2^{k-1}$ and $(2a - b)/2^{k-1}$.  If $M_{k,2}^{-1}
\pi$ were to contain non-integers, then for every permutation $\pi'$ of $S$,
we have that $M_{k,2}^{-1} \pi'$ contains non-integers. But then $S$ is not the
multiset of weight changes for a linear map.  Thus $M_{k,2}^{-1} \pi$ consists
of all integers.  Since $S$ is projection-forcing, those integers are
non-negative.  In particular, $(2a - b)/2^{k-1} \geq 0$, and
hence $b \leq 2a$.

\end{proof}

Now we consider two discrepancies from being constant.

\begin{lemma}\label{lem:linearmapabc}

Let $k \geq 2$, and let $a$, $b$, and $c$ be nonnegative integers.  If
$S$ consists of $2^k - 3$ copies of $a$ and one each of $b$ and $c$
then $S$ is realized by some binary linear map if and only if $a,b,c
\equiv 0 \mod 2^{k-2}$ and either exactly 1 or exactly 3 of $a,b,c$ are
congruent to $0 \mod 2^{k-1}$.

\end{lemma}

\begin{proof}

Suppose $S$ is realized by some binary linear map.  Let $\pi$ be the
permutation of $S$ associated with this map.  Each entry of $M_{k,2}^{-1} \pi$
is one of:

  $$\frac{1}{2^{k-1}}(b + c - a),~~ \frac{1}{2^{k-1}}(a + b -c),~~
       \frac{1}{2^{k-1}}(a + c - b),~~ \frac{1}{2^{k-1}}(3a - b - c).$$
Consider the columns of $M_{k,2}^{-1}$ that are multiplied by $b$ and $c$ in
the product $M_{k,2}^{-1} \pi$.  These columns come from rows of $M_{k,2}$, which in
turn come from linearly independent elements, $s_1$ and $s_2$ of
$\cS_k$.  Since they are linearly independent, $s_1$ and $s_2$ are the
images of the first two columns of $G_{k,2}$ under some automorphism of
$\cS_k$.  By the MacWilliams extension theorem, that automorphism is a
monomial equivalence.  In particular, for each $v \in
\set{(0,1),(1,0),(1,1)}$ there is some coordinate $i$ where the
$i^{th}$ coordinate of $s_j$ is the $j^{th}$ coordinate of $v$.
Thus $(a + b - c)/2^{k-1}$, $(a + c - b)/2^{k-1}$, and
$(b + c - a)/2^{k-1}$ all
appear in any product $M_{k,2}^{-1} \pi$.  Since $\pi$ corresponds to a linear
map, each of these values is an integer.  Adding the first two values,
we have that $2a/2^{k-1}$ is an integer, so $2^{k-2} | a$.
Similarly, $b,c \equiv 0 \mod 2^{k-2}$.  Define $a' = a/2^{k-2}$, $b' =
b/2^{k-2}$, $c' = c/2^{k-2}$.  Then $(1/2)(a' + c' - b')$ is an
integer, so either exactly 1 of a', b', c' is even or they all are.

Conversely, if the parity conditions hold, then let $\pi$ be a permutation
of $S$.  Since $M_{k,2}^{-1} \pi$ consists of integers, we may use this product to
construct a linear map that realizes $S$.

\end{proof}

\begin{proposition}

Let $a$, $b$, and $c$ be nonnegative integers and let $k \geq 3$.  Let
$S$ consist of $2^k - 3$ copies of $a$ and one each of $b$ and $c$.
If $S$ is realized by some linear map then the following are equivalent:

\begin{enumerate}

\item $S$ is $2$-projection-forcing,
\item the set $\set{a,b,c}$ satisfies the triangle inequality and
$3a-b-c\geq 0$,
\item the following four inequalities hold:
$a < b+c+2^{k-1}$,
$b < a+c+2^{k-1}$,
$c < a+b+2^{k-1}$,
$b+c < 3a+2^{k-1}$.

\end{enumerate}

\end{proposition}

\begin{proof}

Clearly, (2) implies (3).  The split difference of $S$ is

$$\delta_2(S) = \begin{cases}
-a+b+c & \mbox{ if } 0 \leq b,c \leq a \\
 a-b+c & \mbox{ if } 0 \leq c \leq a \leq b \\
 a+b-c & \mbox{ if } 0 \leq b \leq a \leq c \\
3a-b-c & \mbox{ if } 0 \leq a \leq b,c.
\end{cases}$$
Thus, Theorem \ref{main1} tells us that (3) implies (1).

We finish by proving (1) implies (2).
Suppose $k \geq 3$ and $S$ is projection-forcing.  If $\pi$ is a
permutation of $S$ then, as in the proof of Proposition
\ref{lem:linearmapabc}, the entries of $M_{k,2}^{-1} \pi$ are
$(b + c - a)/2^{k-1}$, $(a + b - c)/2^{k-1}$,
$(a + c - b)/2^{k-1}$, $(3a - b - c)/2^{k-1}$, and
each of these values appears in $M_{k,2}^{-1}\pi$. If $M_{k,2}^{-1}
\pi$ were to contain non-integers, then for every permutation $\pi'$ of $S$,
we would have that $M_{k,2}^{-1} \pi'$ contains non-integers. But then $S$ would
not be realized by a linear map.  Thus $M_{k,2}^{-1} \pi$
consists of all integers.  Since $S$ is projection-forcing, those
integers are non-negative.  This proves the desired result.

\end{proof}

Using the techniques of this paper, it should be possible to generalize these results to other forms of $S$ and to codes over other finite fields, but the statements quickly become more convoluted.  However, given the very special structure of the matrix $M$, it seems possible -- perhaps probable -- that there is an efficient algorithm to determine whether any given $S$ is projection-forcing.

\section{Acknowledgments}
The authors would like to thank Jamie Radcliffe and the anonymous referees for many helpful comments that improved paper.

\bibliographystyle{plain}
\bibliography{bibtexdatabase2}

\begin{thebibliography}{1}

\bibitem{realspace}
R.~Ahlswede, H.~Aydinian, and L.~Khachatrian.
\newblock Maximum number of constant weight vertices of the unit {$n$}-cube
  contained in a {$k$}-dimensional subspace.
\newblock {\em Combinatorica}, 23(1):5--22, 2003.
\newblock Paul Erd\H os and his mathematics (Budapest, 1999).

\bibitem{bgg1978}
Kenneth Bogart, Don Goldberg, and Jean Gordon.
\newblock An elementary proof of the {M}ac{W}illiams theorem on equivalence of
  codes.
\newblock {\em Information and Control}, 37(1):19--22, 1978.

\bibitem{bk07}
Joshua Brown~Kramer.
\newblock {\em Two Problems in Extremal Combinatorics}.
\newblock PhD thesis, University of Nebraska - Lincoln, 2007.

\bibitem{hp2003}
W.~Cary Huffman and Vera Pless.
\newblock {\em Fundamentals of error-correcting codes}.
\newblock Cambridge University Press, Cambridge, 2003.

\bibitem{MacWilliams}
F.~J. MacWilliams.
\newblock {\em Combinatorial problems of elementary abelian groups}.
\newblock PhD thesis, Harvard University, 1962.

\end{thebibliography}

\end{document}